\documentclass[a4paper]{amsart}
\usepackage[T1]{fontenc}
\usepackage{amscd}
\usepackage{pstricks}
\usepackage{hyperref}

\title[Vector bundles over Klein surfaces]{Moduli spaces of vector bundles over\\ a Klein surface}
\author{Florent Schaffhauser}
\curraddr{Department of Mathematics, University of Los Andes, Bogota, Colombia.}
\address{IHES, 35 route de Chartres, 91440 Bures-sur-Yvette, France.}
\email{florent@ihes.fr}

\newtheorem{theorem}{Theorem}[section]
\newtheorem{lemma}[theorem]{Lemma}
\newtheorem{proposition}[theorem]{Proposition}
\newtheorem{corollary}[theorem]{Corollary}
\newtheorem{definition}[theorem]{Definition}

\newtheorem*{thm*}{Theorem}

\theoremstyle{definition}
\newtheorem*{ack}{Acknowledgments}

\newcommand{\C}{\mathbb{C}}
\newcommand{\R}{\mathbb{R}}
\newcommand{\Z}{\mathbb{Z}}

\newcommand{\Vect}{\mathrm{Vect}}
\newcommand{\Mod}{\mathcal{M}^{\, r,d}_{(M,\sigma)}}
\newcommand{\sectionsofE}{\Omega^0(M;E)}
\newcommand{\oneformsofE}{\Omega^1(M;E)}
\newcommand{\twoformsofE}{\Omega^2(M;E)}
\newcommand{\kformsofE}{\Omega^k(M;E)}

\newcommand{\oneformsEzero}{\Omega^1(M;E)}
\newcommand{\zerooneforms}{\Omega^{0,1}(M;E)}
\newcommand{\zerooneendom}{\Omega^{0,1}(M;\mathfrak{gl}(E))}

\newcommand{\functions}{\Omega^0(M;\mathbb{C})}
\newcommand{\zeroonematrices}{\Omega^{0,1}(U_{\tau};\mathfrak{gl}_r(\C))}

\newcommand{\Ut}{U_{\tau}}
\newcommand{\Utp}{U_{\tau'}}
\newcommand{\Uttp}{U_{\tau} \cap U_{\tau'}}
\newcommand{\Vt}{V_{\tau}}
\newcommand{\Vtp}{V_{\tau'}}
\newcommand{\Vttp}{V_{\tau} \cap V_{\tau'}}
\newcommand{\Wt}{W_{\tau}}
\newcommand{\Wtp}{W_{\tau'}}

\newcommand{\pt}{\varphi_{\tau}}
\newcommand{\ptp}{\varphi_{\tau'}}
\newcommand{\pttp}{\varphi_{\tau}\circ\varphi_{\tau'}^{-1}}
\newcommand{\zt}{z_{\tau}}
\newcommand{\ztp}{z_{\tau'}}
\newcommand{\zttp}{z_{\tau}\circ z_{\tau'}^{-1}}
\newcommand{\http}{h_{\tau\tau'}}
\newcommand{\gttp}{g_{\tau\tau'}}
\newcommand{\psittp}{\psi_{\tau}\circ\psi_{\tau'}^{-1}}
\newcommand{\Bt}{B_{\tau}}
\newcommand{\Btp}{B_{\tau'}}
\newcommand{\Lt}{\lambda_{\tau}}

\newcommand{\At}{A_{\tau}}
\newcommand{\Atp}{A_{\tau'}}

\newcommand{\ut}{u_{\tau}}
\newcommand{\utp}{u_{\tau'}}
\newcommand{\antiHermonematrices}{\Omega^1(U_{\tau};\mathfrak{u}_r)}
\newcommand{\st}{s_{\tau}}
\newcommand{\w}{\omega}
\newcommand{\GL}{\mathrm{GL}}
\newcommand{\gl}{\mathfrak{gl}}
\newcommand{\fraku}{\mathfrak{u}}
\newcommand{\G}{\mathcal{G}}

\newcommand{\sig}{\sigma}
\newcommand{\sigt}{\widetilde{\sigma}}
\newcommand{\calL}{\mathcal{L}^{r,d}_{\sigt}}
\newcommand{\calN}{\mathcal{N}^{r,d}_{\sigma}}
\newcommand{\bs}{\beta_{\sigma}}

\newcommand{\calH}{\mathcal{H}}
\newcommand{\calF}{\mathcal{F}}
\newcommand{\calD}{\mathcal{D}}
\newcommand{\ov}[1]{\overline{#1}}
\newcommand{\os}[1]{\overline{\sigma^*#1}}
\newcommand{\db}{\overline{\partial}}
\newcommand{\del}{\partial}
\newcommand{\ttp}{\tau\tau'}
\newcommand{\conn}{\mathcal{A}(E,h)}

\newcommand{\psit}{\psi_{\tau}}
\newcommand{\quot}{\mathcal{A}(E,h)/\negthickspace /_{*i2\pi\frac{d}{r} \mathrm{\Id}_E} \mathcal{G}_{(E,h)}}

\newcommand{\unitarygaugegp}{\G_{(E,h)}}
\newcommand{\cxgaugegp}{\G^{\, \C}_{(E,h)}}

\newcommand{\sections}{\Omega^0(M;\mathfrak{u}(E,h))}
\newcommand{\oneforms}{\Omega^1(M;\mathfrak{u}(E,h))}
\newcommand{\twoforms}{\Omega^2(M;\mathfrak{u}(E,h))}
\newcommand{\tmu}{i2\pi \mu\, \mathrm{\Id}_{E}}
\newcommand{\antiHermsections}{\Omega^0(M;\mathfrak{u}(E,h))}
\newcommand{\antiHermoneforms}{\Omega^1(M;\mathfrak{u}(E,h))}
\newcommand{\antiHermtwoforms}{\Omega^2(M;\mathfrak{u}(E,h))}
\newcommand{\antiHermendom}{\mathfrak{u}(E,h)}

\newcommand{\Aut}{\mathrm{Aut}}
\newcommand{\End}{\mathrm{End}}

\newcommand{\Id}{\mathrm{Id}}
\newcommand{\tr}{\mathrm{tr}}
\newcommand{\U}{\mathbf{U}}
\newcommand{\Ad}{\mathrm{Ad}}
\newcommand{\Fix}{\mathrm{Fix}}
\newcommand{\asigt}{\alpha_{\sigt}}
\newcommand{\vol}{\mathrm{vol}}
\newcommand{\Dol}{\mathrm{Dol}}
\newcommand{\degr}{\mathrm{deg}}
\newcommand{\rk}{\mathrm{rk}}
\newcommand{\gr}{\mathrm{gr}}
\newcommand{\fibre}{F^{-1}\big(\{*i2\pi\frac{d}{r}\, \mathrm{Id}_E\}\big)}
\newcommand{\realgaugegp}{\unitarygaugegp^{\sigt}}
\newcommand{\Ms}{M^{\sigma}}
\newcommand{\Es}{E^{\sigt}}
\newcommand{\calC}{\mathcal{C}}
\renewcommand{\phi}{\varphi}
\renewcommand{\mod}{\mathrm{mod\ }}
\newcommand{\calE}{\mathcal{E}}

\begin{document}

\begin{abstract}
A compact topological surface $S$, possibly non-orientable and with non-empty boundary, always admits a Klein surface structure (an atlas whose transition maps are dianalytic). Its complex cover is, by definition, a compact Riemann surface $M$ endowed with an anti-holomorphic involution which determines topologically the original surface $S$. In this paper, we compare dianalytic vector bundles over $S$ and holomorphic vector bundles over $M$, devoting special attention to the implications that this has for moduli varieties of semistable vector bundles over $M$. We construct, starting from $S$, totally real, totally geodesic, Lagrangian submanifolds of moduli varieties of semistable vector bundles of fixed rank and degree over $M$. This relates the present work to the constructions of Ho and Liu over non-orientable compact surfaces with empty boundary (\cite{Ho-Liu_YM}).
\end{abstract}

\maketitle

\begin{center}
\textit{Dedicated to the memory of Paulette Libermann (1919-2007).}
\end{center}

Let $(M,\sigma)$ be a compact Riemann surface endowed with an anti-holomorphic involution, and let $\Mod$ be the moduli variety of semistable holomorphic vector bundles of rank $r$ and degree $d$ over $M$.
The goal of the paper is to show that the set $$\calN:=\{[\gr(\calE)]\in\Mod : \calE\ \mathrm{is\ semistable\ and\ real}\}$$ of moduli of real semistable bundles of rank $r$ and degree $d$ is a totally real, totally geodesic, Lagrangian submanifold of $\Mod$. By a real vector bundle over $M$, we mean a holomorphic vector bundle $\calE$ over $M$ which is a \textit{real bundle over} $(M,\sigma)$ in the sense of Atiyah (see \cite{Atiyah_real_bundles} and definition \ref{def_real_bundle}). We denote $\vol_M$ the unique volume form such that $\int_M \vol_M =1$ in the canonical orientation of the Riemann surface $M$, and we assume throughout that $\sigma$ is an isometry satisfying $\sig^*\vol_M=-\vol_M$.\\ Moduli varieties of semistable vector bundles over orientable surfaces with empty boundary are central objects of attention in algebraic geometry, and the study of analogous spaces for other types of topological surfaces (non-orientable or with non-empty boundary) has attracted the attention of many in recent years, with an emphasis on the case of non-orientable surface with empty boundary (see \cite{Ho-Liu_YM}, \cite{Baird}, \cite{Ramras} for bundles over non-orientable surfaces with empty boundary, and \cite{Sav-Wang} for bundles over a sphere with a finite number of open disks removed). In \cite{Ho-Liu_YM}, Ho and Liu gave a complete theory for the case of flat bundles over a non-orientable surface with empty boundary. As we shall see shortly, their point of view is similar to ours, in that they think of a bundle over a non-orientable surface (with empty boundary) as a bundle over the orientation cover of that surface with an additional $\Z/2\Z$ symmetry. In the present work, systematic use is made of a fixed Klein surface structure on a given compact surface, which enables us, replacing the orientation cover with the complex cover, to handle the case of non-orientable surfaces with non-empty boundary (the orientation cover of a surface with boundary still has a boundary, whereas its complex cover does not, so the Atiyah-Bott theory can always be applied to the complex cover). The use of complex covers also produces involutions of the various moduli varieties that we consider slightly different from the ones of Ho and Liu. On the set of isomorphism classes of smooth complex vector bundles for instance, what we are doing amounts to replacing $E\mapsto \sigma^*E$ with $E\mapsto \os{E}$. This apparently innocuous change has the merit of giving an involution of each of the moduli varieties $\Mod$, even when $d\not=0$, in contrast to what has appeared so far.\\ The gist of the construction is as follows. We shall prove that the set $\calN$ of moduli of real semistable bundles of rank $r$ and degree $d$ is a union of connected components of the fixed-point set of an involution $\bs$ of $\Mod$. To see that $\calN$ is a totally real, totally geodesic, Lagrangian submanifold of $\Mod$, it then suffices to prove that $\bs$ is an anti-symplectic isometry. In section \ref{intro}, motivation is drawn from the fact that $\calN$ may be thought of as a moduli space of dianalytic bundles over the Klein surface $S:=M/\sigma$. This explains the role played by involutions in the paper. To actually define the involution $\bs$ and perform the practical computations, we recall in section \ref{moduli_of_ss_bundles} the presentation of $\Mod$ as a K\"ahler quotient, which is based on Donaldson's characterisation of stable bundles in terms of unitary connections. It is then a simple matter to prove the things that we have asserted in this introduction. The main result of the paper is theorem \ref{lag_submanifold}, stated at the end of the last section.\\ As a final comment before entering the subject, we should mention that, when we speak of a Lagrangian submanifold of $\Mod$, one should understand the following. When $r\geq 1$ and $d\in \Z$ are coprime, every semistable bundle is in fact stable and $\Mod$ is a smooth, projective variety over $\C$, so it makes sense to speak of a Lagrangian submanifold of $\Mod$. When $r$ and $d $ are not coprime, which includes the important case $d=0$, there are two ways of thinking about theorem \ref{lag_submanifold}. Either by restricting one's attention to the moduli variety of stable bundles, which is a smooth, quasi-projective, complex variety and also a K\"ahler quotient, and then the theorem remains true provided one replaces the word semistable with the word stable, with the same proof. Or by recalling that when $\Mod$ is not smooth, it has a statified symplectic structure in the sense of Lerman and Sjamaar (\cite{LS}), and then $\calN$ is seen to be a stratified subspace of that moduli space, whose intersection with a given symplectic stratum is Lagrangian. The same goes for the property of being totally real, or totally geodesic (it is a stratum by stratum statement). Finally, let us mention that, in the paper, when we need to distinguish between a holomorphic vector bundle and its underlying smooth complex vector bundle, we denote the former by $\calE$ and the latter by $E$.

\section{Motivation: dianalytic vector bundles over a Klein surface}\label{intro}

We start by recalling the fundamentals of Klein surfaces (\textit{e.g.} \cite{AG}). A function $f:U\to\C$ defined on an open subset of $\C$ is called \textbf{dianalytic in} $U$ if it is either holomorphic or anti-holomorphic (meaning respectively that $\db{f}=0$ and that $\del{f}=0$) on any given connected component of $U$. Denote $\ov{\calH}=\{z\in\C\ |\ Im\, z \geq 0\}$ the closed upper half-plane of $\C$, with the induced topology, and denote $\calH$ the open upper half-plane. Let $U$ be an open subset of $\ov{\calH}$. A function $f:U\to\C$ is called \textbf{dianalytic on} $U$ if it is continuous on $U$, dianalytic in $U\cap \calH$ and satisfies $f(U\cap\R)\subset \R$. We can consider the restrictions of such functions to any open subset $V$ of $\ov{\calH}$ such that $V\subset U$ and this defines a sheaf of functions on any open subset $U\subset\ov{\calH}$, denoted $\calD_U$, called the sheaf of dianalytic functions on $U$. A \textbf{Klein surface} is a topological space $S$ together with a subsheaf of the sheaf of continuous functions making $S$ locally isomorphic as a ringed space to a space $(U,\calD_U)$ for some open subset $U$ of $\ov{\calH}$. In other words, $S$ is a topological surface (possibly non-orientable and with non-empty boundary) which admits a covering by open sets homeomorphic to open subsets of $\ov{\calH}$ and such that the associated transition maps are dianalytic. A homomorphism between two Klein surfaces $S_1$ and $S_2$ is a continuous mapping $f:S_1\to S_2$ which is dianalytic in local charts (in particular, it sends the boundary of $S_1$ to the boundary of $S_2$). One may observe that the topological requirement that the boundary of $S_1$ should be sent to the boundary of $S_2$ is the reason for the condition $f(U\cap\R)\subset \R$ in the definition of a dianalytic function on an open subset $U$ of $\ov{\calH}$.\\ When the open set $U\subset\ov{\calH}$ is connected, then, by the Schwarz reflection principle, the functions $f\in\calD_{\ov{\calH}}(U)$ are in bijective correspondence with the holomorphic functions $F$ defined on the symmetric open set $U\cup\ov{U}$ of $\C$ satisfying $F|_{U}=f$ and $F(\ov{z})=\ov{F(z)}$ for all $z\in U\cup\ov{U}$ (in particular, $F$ sends $U\cap\R$ to $\R$). Observe that $U\cup\ov{U}$ is connected if, and only if, $U\cap\R$ is non-empty, and that this is the case where we actually need the Schwarz reflection principle. This property is the basis for the existence of the \textbf{complex cover} of a Klein surface $S$, which, by definition, is a Riemann surface $M$ together with an anti-holomorphic $\sigma$ and a map $p:M\to S$ inducing a homeomorphism from $M/\sigma$ to $S$ (in fact, an isomorphism of Klein surfaces, for the natural Klein structure on $M/\sigma$). If $(M',\sigma',p')$ is another complex cover of $S$, then there is a biholomorphic bijection $f:M\to M'$ such that $\sigma'= f\sigma f^{-1}$ (see \cite{AG}, section 1.6), which is why we shall henceforth speak of \textit{the} complex cover of $S$. An important point for us in the study of vector bundles on compact topological surfaces \textit{which do not admit a Riemann surface structure} is that any compact topological surface (possibly non-orientable and with non-empty boundary) may be endowed with a Klein surface structure (see \cite{AG}, section 1.5). We now briefly recall the construction of the complex cover of a Klein surface, as it is instructive in order to, later, pull back dianalytic bundles over $S$ to holomorphic bundles over $M$. Let $(\Ut,\pt:\Ut\overset{\simeq}{\longrightarrow} \Vt\subset \ov{\calH})_{\tau\in T}$ be a dianalytic atlas of $S$ and form the open sets $\Wt = \Vt \cup \ov{\Vt} \subset \C$ and the topological space $\Omega:=\bigsqcup_{\tau\in T} \Wt$, with the final topology (each $\Wt$ is open in $\Omega$). For simplicity, we assume that $\Ut\cap\Utp$ is connected for all $\tau,\tau'$. As $(\Ut,\pt)_{\tau\in T}$ is a dianalytic atlas, the transition map $\pt\circ\ptp^{-1}$ from $\Vtp$ to $\Vt$ is either holomorphic or anti-holomorphic. We first treat the case where it is anti-holomorphic. In this case, we glue $\ptp(\Uttp) \subset \Vtp$ to $\ov{\pt}(\Uttp)\subset \ov{\Vt}$, and $\ov{\ptp}(\Uttp) \subset \ov{\Vtp}$ to $\pt(\Uttp) \subset \Vt$.

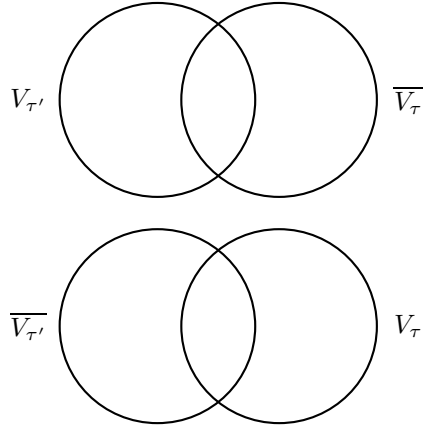
\begin{figure}[ht]
\centerline{
\begin{pspicture}(-3,-3)(3,3)
\pscircle(-0.8,1.5){1.3}
\rput(-2.5,1.5){$V_{\tau'}$}
\pscircle(0.8,1.5){1.3}
\rput(2.5,1.5){$\overline{V_{\tau}}$}
\pscircle(-0.8,-1.5){1.3}
\rput(-2.5,-1.5){$\overline{V_{\tau'}}$}
\pscircle(0.8,-1.5){1.3}
\rput(2.5,-1.5){$V_{\tau}$}
\end{pspicture}
}
\caption{The open set $[W_{\tau}]\cup[W_{\tau'}]$ of $M$ when $\varphi_{\tau}\circ\varphi_{\tau'}^{-1}$ is anti-holomorphic.}
\end{figure}
\noindent Had $\pttp$ been holomorphic, we would have glued $\ptp(\Uttp)\subset\Vtp$ to $\pt(\Uttp)\subset\Vt$ and $\ov{\ptp}(\Uttp)\subset\ov{\Vtp}$ to $\ov{\pt}(\Uttp)\subset\ov{\Vt}$. This defines an equivalence relation $\mathcal{R}$ on $\Omega$ and we set $M:=\Omega/\mathcal{R}$, with the quotient topology. Let $[\Wt]$ denote the image of $\Wt\subset\Omega$ under the canonical projection and consider the map $$\zt: \begin{array}{rcl} [\Wt] & \longrightarrow & \Wt\\ \lbrack x \rbrack & \longmapsto & x\end{array}.$$ Each $\zt$ is a homeomorphism and the transition maps $\zttp$ are holomorphic. Indeed, if for instance $\pttp$ is anti-holomorphic, then, by definition of the gluing on $\Omega$, the map $\zttp:\ztp([\Wt]\cap[\Wtp]) \to \zt([\Wt]\cap [\Wtp])$ is the map $$v\longmapsto \left\{\begin{array}{rcl}  \ov{\pttp}(v) & \mathrm{if} & v\in\Vtp,\\ \pttp(\ov{v}) & \mathrm{if} & v\in\ov{\Vtp}\end{array},\right.$$ and this map is holomorphic (this requires the use of the Schwarz reflection principle precisely when $\Vttp$ intersects $\R$ in $\ov{\calH}$). The case where $\pttp$ is holomorphic is similar. So we have a holomorphic atlas $([\Wt],\zt)_{\tau \in T}$ on $M$. The anti-holomorphic involution $\sigma$ on $M$ is just complex conjugation in the local charts $([\Wt],\zt)$ and the projection $p:M\to S$ takes $[x]=[a+ib]\in[\Wt]$ to $a+i|b|\in\Vt$.\\ We now consider a fixed Klein surface $S$ and denote $(M,\sigma,p)$ its complex cover. As earlier, we denote $(\Ut,\pt)_{\tau\in T}$ a dianalytic atlas of $S$ and set $\Vt=\pt(\Ut)\subset\ov{\calH}$.
\begin{definition}[Dianalytic vector bundle]
A \textbf{dianalytic vector bundle} $Q$ over $S$ is a topological complex vector bundle for which there exists a family $(\Ut,\psit)_{\tau\in T}$ of local trivialisations whose associated $1$-cocycle of transition maps, denoted $$\big(\http: \Uttp \to \GL_r(\C)\big)_{\tau,\tau'},$$ is dianalytic. In particular, the map $\http$ sends $(\Uttp)\cap\partial{S}$ to $\GL_r(\R)$.
\end{definition}
\noindent Saying that the map $\http:\Uttp \to \GL_r(\C)\subset\gl_r(\C)$ is dianalytic means, by definition, that its $r^2$ components are either simultaneously holomorphic or simultaneously anti-holomorphic. This implies that a dianalytic bundle is in particular a dianalytic manifold. Accordingly, we define a homomorphism between two dianalytic bundles $Q$ and $Q'$ over $S$ to be a homomorphism $f:Q\to Q'$ of topological complex vector bundles which is dianalytic in local charts. To clarify the notation, let us mention that we think of $\http$ as the transition map from $\Utp$ to $\Ut$, associated to the map $\psittp$ by the relation $$\psittp(x,\eta) = \big(x,\http(x).\eta\big)$$ for all $(x,\eta)\in(\Uttp)\times\C^r$. We shall now show that the pulled back bundle $(p^*Q\to M)$ is a holomorphic bundle. First, a definition. If $(E\to M)$ is any complex vector bundle over a topological space $M$, the vector bundle $(\ov{E}\to M)$ is defined to be the complex vector bundle over $M$ whose fibres are the fibres of $E$ with the complex structure defined by multiplication by $-i$ (called the conjugate complex structure). Alternatively, if $(\gttp)_{\tau,\tau'}$ is a \textit{smooth} $1$-cocycle representing $E$, then $\ov{E}$ is represented by $(\ov{\gttp})_{\tau,\tau'}$. The choice of a Hermitian metric on $E$ gives a \textit{complex linear} isomorphism $\ov{E} \simeq E^*$. In particular, $\degr(\ov{E}) = -\degr(E)$. We now proceed with the construction of $p^*Q$. Consider the local charts $Q|_{\Ut}\simeq \Vt \times \C^r$ of $Q$, and form the product bundles $\Vt \times \C^r$ and $\ov{\Vt} \times \C^r$. The transition maps of $Q$ are the maps $$\begin{array}{rcl} \ptp(\Uttp) \times \C^r & \longrightarrow & \pt(\Uttp)\times\C^r \\ (v,\eta) & \longmapsto & \big(\pttp(v), \http\big(\ptp^{-1}(v)\big).\eta\big)\end{array}.$$ By definition of a cocycle of transition maps, the dianalyticity requirement on $\http$ means that $\pttp$ and $\http\circ\ptp^{-1}$ are either simultaneously holomorphic or simultaneously anti-holomorphic. Let us first consider the case where they are both anti-holomorphic. Then, as in the construction of the complex cover of $S$, we set, for $(w,\eta)\in \ztp([\Wt]\cap[\Wtp])\times\C^r$, $$\big(w,\eta\big) \sim \left\{\begin{array}{rcl} \big(\ov{\pttp}(w),\ \ov{h_{\tau\tau'}\circ\ptp^{-1}}(w).\eta\big) & \mathrm{if} & w\in\Vtp, \\ \big(\pttp(\ov{w}),\ h_{\tau\tau'}\circ\ptp^{-1}(\ov{w}).\eta\big) & \mathrm{if} & w\in\ov{\Vtp}\end{array}.\right.$$ Had $\pttp$ and $\http\circ\ptp^{-1}$ been holomorphic, we would have set $$\big(w,\eta\big) \sim \left\{\begin{array}{rcl} \big(\pttp(w),\ h_{\tau\tau'}\circ\ptp^{-1}(w).\eta\big) & \mathrm{if} & w\in\Vtp, \\ \big(\ov{\pttp}(\ov{w}),\ \ov{h_{\tau\tau'}\circ\ptp^{-1}}(\ov{w}).\eta\big) & \mathrm{if} & w\in\ov{\Vtp}\end{array}.\right.$$ Note that the elements on the right are elements of $\zt([\Wt]\cap[\Wtp])\times\C^r$. This defines an equivalence relation on the bundle $\sqcup_{\tau\in T} ([\Wt]\times\C^r) \to \sqcup_{\tau\in T} [\Wt]$, therefore a topological complex vector bundle $E$ with typical fibre $\C^r$ over $M$ (the complex cover of $S$) with a $1$-cocycle of transition maps $$\big(\gttp:[\Wt]\cap[\Wtp] \longrightarrow \GL_r(\C)\big)_{\tau,\tau'}$$ satisfying $$\gttp\circ\ztp^{-1} (w)= \left\{ \begin{array}{rcl} \ov{\http\circ\ptp^{-1}}(w) & \mathrm{if} & w\in \Vtp \\ \http\circ\ptp^{-1} (\ov{w}) & \mathrm{if} & w\in\ov{\Vtp} \end{array}\right.$$ when $\pttp$ and $\http\circ\ptp^{-1}$ are anti-holomorphic, and $$\gttp\circ\ztp^{-1}(w) = \left\{ \begin{array}{rcl} \http\circ\ptp^{-1}(w) & \mathrm{if} & w\in \Vt \\ \ov{\http\circ\ptp^{-1}} (\ov{w}) & \mathrm{if} & w\in\ov{\Vtp} \end{array}\right.$$ when $\pttp$ and $\http\circ\ptp^{-1}$ are holomorphic. In particular, $\gttp\circ\ztp^{-1}$ is always a holomorphic map (this requires the use of the Schwarz reflection principle precisely when $\Vttp$ intersects $\R$ in $\ov{\calH}$). So $(\gttp)_{\tau,\tau'}$ is a holomorphic cocycle for the holomorphic atlas $([\Wt],\zt)_{\tau\in T}$ of $M$ and the complex vector bundle it defines is indeed holomorphic. By construction, the bundle $\calE$ thus obtained is no other than $p^*Q$, where $p:M\to S$ is the (ramified) covering map between $M$ and $S$. The construction also shows that we have a commutative diagram \begin{equation*}\begin{CD} [\Wt]\times\C^r @>\widetilde{\sigma}>> [\Wt]\times\C^r \\ @VVV @VVV \\ [\Wt] @>\sigma>> [\Wt]=\sigma([\Wt]), \end{CD}\end{equation*} where $\widetilde{\sigma}$ is fibrewise $\eta\mapsto \ov{\eta}$ and covers $\sigma$. This means that $\calE$ has a global involution $\widetilde{\sigma}:\calE\to \calE$ covering $\sigma$ and which is $\C$-antilinear in the fibres, so $\calE=p^*Q$ is a \textbf{real bundle} over the \textbf{real space} $(M,\sigma)$ in the sense of Atiyah (\cite{Atiyah_real_bundles}).
\begin{definition}[Real bundle over $M$]\label{def_real_bundle}
Let $M$ be a Riemann surface and let $\sigma$ be an anti-holomorphic involution of $M$. A \textbf{real holomorphic vector bundle} $\calE\to M$ is a holomorphic vector bundle, together with an involution $\widetilde{\sigma}$ of $\calE$ making the diagram \begin{equation*}\begin{CD} \calE @>\widetilde{\sigma}>> \calE \\ @VVV @VVV \\ M @>\sigma>> M \end{CD}\end{equation*} commutative, and such that, for all $x\in M$, the map $\widetilde{\sigma}|_{\calE_x}: \calE_x \to \calE_{\sigma(x)}$ is $\C$-antilinear: $$\widetilde{\sigma}(\lambda\cdot\eta)=\ov{\lambda}\cdot\widetilde{\sigma}(\eta),\ \mathrm{for\ all}\ \lambda\in\C\ \mathrm{and\ all}\ \eta\in E_x.$$ A homomorphism between two real bundles $(\calE,\widetilde{\sigma})$ and $(\calE',\widetilde{\sigma}')$ is a homomorphism $$f:\calE\to \calE'$$ of holomorphic vector bundles over $M$ such that $f\circ\widetilde{\sigma} = \widetilde{\sigma}' \circ f$.
\end{definition}
\noindent We shall often call $\sigt$ the \textbf{real structure} of $\calE$. If clear from the context, it is convenient, following the convention in \cite{Atiyah_real_bundles}, to simply write $x\mapsto \ov{x}$ for $\sig$, and $v\mapsto\ov{v}$ for $\sigt$. There are similar notions of real structures for topological, smooth, and Hermitian complex vector bundles. For Hermitian bundles for instance, one requires that $\sigt$ should be a $\C$-antilinear, involutive isometry which covers $\sigma$. A \textbf{real section} of a real bundle $\calE$ over $M$ is a section $s$ of $\calE$ satisfying $s(\ov{x})=\ov{s(x)}$ for all $x\in M$. By the Schwarz reflection principle, the real sections of the real bundle $p^*Q$ constructed above are in one-to-one correspondence with the dianalytic sections of $Q$. This implies that, if two dianalytic bundles $Q$ and $Q'$ over $S$ satisfy $p^*Q\simeq p^*Q'$ as real vector bundles over $M$, then $Q\simeq Q'$ as dianalytic vector bundles over $S=M/\sigma$. We sum up this discussion with the following lemma.
\begin{lemma}
Let $S$ be a Klein surface and let $(M,\sigma,p)$ be its complex cover. Then the pulling back functor $p^*$ is an equivalence of categories between the category of dianalytic vector bundles over $S=M/\sig$, and the category of  real vector bundles over $M$.
\end{lemma}
\noindent Thus, instead of studying dianalytic vector bundles over a Klein surface, one may just as well choose to study real vector bundles over the complex cover of that Klein surface. Observe that these two equivalent categories are equivalent to a third one, the category of real algebraic vector bundles over a real algebraic curve (hence the term, real bundle). Indeed, a real algebraic curve $X/\R$ gives rise to a Riemann surface $M=X(\C)$ with anti-holomorphic involution $\sigma$ induced by the complex conjugation of $\C$, and complexification of real algebraic bundles over $X$ provides an equivalence of categories between the category of real algebraic bundles over $X$ and the category of real holomorphic bundles over $M$ (see \cite{Atiyah_real_bundles}, page 370). Vector bundles over real algebraic curves are studied in \cite{BHH}, where results similar to ours are obtained. Indeed, one may observe that a Klein surface $M/\sigma$, together with its sheaf of dianalytic functions, naturally gives rise to a scheme of finite type over $\R$. The set of real points of that scheme is in bijection with the boundary of $M/\sigma$.\\ The point that we try to make in the rest of the paper is that we can (almost) think of the set of real vector bundles over $M$ as the fixed-point set of an involution of a larger set, morally the set of isomorphism classes of all holomorphic vector bundles. Certainly, a real bundle over $M$ satisfies $\os{\calE}\simeq \calE$ (the isomorphism being induced by $\widetilde{\sigma}$). But the converse is not true in general (indeed, quaternionic bundles, meaning those which admit a $\C$-antilinear automorphism $\sigt$ such that $\sigt^2=-\Id_\calE$, are also fixed by that involution), so we cannot just consider the bundles whose isomorphism class is fixed under the involution $\calE\mapsto\os{\calE}$ if we want the real bundles to stand out. One may observe that restricting the involution to the set of isomorphism classes of stable bundles is not sufficient for our purposes, again because of the possibility that a stable bundle satisfying $\os\calE\simeq \calE$ might be quaternionic. Instead, we use the differential geometric approach of Atiyah, Bott and Donaldson, which consists in replacing holomorphic vector bundles of a given topological type with unitary connections on a fixed Hermitian bundle with that same topological type. Then, for each choice of a real Hermitian structure $\sigt$ on that Hermitian bundle, we construct an involution $\asigt$ on the space of connections, which induces $\bs:[\calE]\mapsto [\os{\calE}]$ on isomorphism classes of stable bundles, regardless of the choice of $\sigt$. The upshot of that construction is that the connections which are fixed by the various $\asigt$ thus constructed, are exactly those which define the real holomorphic bundles. In a forthcoming paper, we shall show that quaternionic bundles can also be characterised by means of an involution and use this, combined with the real case dealt with in the present work, to derive topological information on $\Fix(\bs)$, such as a bound on the number of connected components of that set.

\section{Moduli spaces of semistable bundles over a Riemann surface}\label{moduli_of_ss_bundles}

In this section, we recall the basics about moduli varieties of semistable holomorphic vector bundles over a compact Riemann surface $M$ with a fixed compatible Riemannian metric normalised to unit volume. The section contains nothing new and the specialist reader may skip altogether. Although some aspects of what we shall say extend to higher-dimensional compact K\"ahler manifolds, the benefit of working in complex dimension one is twofold. \begin{enumerate} \item The classification of smooth complex vector bundles over a compact Riemann surface $M$ is very simple to state; such a vector bundle $E$ is classified by two integers, its rank $r$ and its degree $d:=\int_M c_1(E) \in \Z$, where $c_1(E)\in H^2(M;\Z)$ is the first Chern class of $E$, viewed as a real-valued differential $2$-form on $M$. \item When studying holomorphic structures on a smooth complex vector bundle $E$ over $M$, we do not have to worry about integrability conditions; isomorphism classes of holomorphic structures on $E$ are in bijective correspondence with orbits of Dolbeault operators on $E$, as we shall recall shortly.\end{enumerate} References for the present section are sections $7$ and $8$ of \cite{AB}, section $2$ of \cite{Don_NS}, and sections $2.1$ and $2.2$ of \cite{DK}. Although the involution $\sigma$ of $M$ plays no role in this section, we denote $\Mod$ the moduli variety of semistable vector bundles of rank $r$ and degree $d$ on $M$, to avoid introducing another piece of notation.

\subsection{Holomorphic vector bundles and unitary connections}\label{hol_bundles_and_unitary_connections}

\begin{definition}[Dolbeault operator]
A \textbf{Dolbeault operator} $D$ on a smooth complex vector bundle $(E\to M)$ is a $\C$-linear map $$D:\sectionsofE \longrightarrow \zerooneforms$$ such that, for any smooth function $f:M\to\C$ and any smooth section $s\in\sectionsofE$, $$D(fs) = (\db{f})s + f(Ds)\ (Leibniz\ identity).$$ 
\end{definition}
\noindent The $\db$ in the definition above is the usual Cauchy-Riemann operator on $\functions$, taking a smooth function $f:M\to\C$ to the $\C$-antilinear part of $df\in\Omega^1(M;\C)$. In particular, $\db{f}=0$ if, and only if, $f$ is holomorphic (\textit{i.e.} the Cauchy-Riemann equations are satisfied). A Dolbeault operator on $E$ is also called a (0,1)-connection on $E$. In a local trivialisation $(U_{\tau},\psit)$ of the bundle, the Dolbeault operator $D$ takes the form $$(Ds)_{\tau} = \db{s_{\tau}} + B_{\tau} s_{\tau},$$ where $B_{\tau}\in\zeroonematrices$ and $\db{}$ acts componentwise on sections of $\psit : E_{\tau}\overset{\simeq}{\longrightarrow} U_{\tau} \times \C^r$. Denote $(\Ut,\psit)_{\tau\in T}$ a family of local trivialisations of $E$ which covers $M$ and let $$(g_{\ttp}: \Ut\cap\Utp \longrightarrow \GL_r(\C))_{\tau,\tau'}$$ be the corresponding smooth $1$-cocycle of transition maps. With our conventions on $1$-cocycles of transition maps, the family $(s_{\tau})_{\tau\in T}$ satisfies $s_{\tau}=g_{\ttp}s_{\tau'}$, and the family $(\Bt)_{\tau\in T}$ satisfies \begin{equation}\label{Dolbeault_compatibility} \Bt = g_{\ttp} \Btp g_{\ttp}^{-1} - (\db{g_{\ttp}}) g_{\ttp}^{-1}.\end{equation} Conversely, a Dolbeault operator is completely determined by such a family $B:=(\Bt)_{\tau\in T}$. In the following, we shall denote $\db_{B}$ the Dolbeault operator corresponding to a family $B=(\Bt)_{\tau\in T}$ satisfying (\ref{Dolbeault_compatibility}). It follows from the Leibniz identity that the difference $\db_{B}-\db_{B'}$ of two Dolbeault operators is a $\functions$-linear map from $\sectionsofE$ to $\zerooneforms$, so it corresponds to an element of the space $\zerooneendom$ of (0,1)-forms over $M$ with values in the bundle $\gl(E)$ of endomorphisms of $E$. Consequently, the set $\Dol(E)$ of all Dolbeault operators on $E$ is an infinite-dimensional complex affine space whose group of translations is $\zerooneendom$. If $u\in\Aut(E)$ is an automorphism of $E$, the operator $$\db_{u(B)}s := u\big(\db_B(u^{-1}s)\big)$$ is a Dolbeault operator on $E$ and this defines an action of $\Aut(E)$ on $\Dol(E)$ (observe that we denoted $\db_{u(B)}$ the result of the action of $u$ on $\db_B$). If we still denote $\db_B$ the Dolbeault operator on $\gl(E)$ induced by the operator $\db_B$ on $E$ (explicitly, it is the operator defined by the family $(\mathrm{ad}\, \Bt)_{\tau\in T}$, which satisfies relation (\ref{Dolbeault_compatibility}) when the cocycle $(\Ad\, g_{\ttp})_{\tau,\tau'}$ is used to represent $\gl(E)$) and expand the above formula using the Leibniz rule, we obtain the following well-known formula for the above action, $$\db_{u(B)} = \db_B - (\db_{\mathrm{ad} B} u)u^{-1}.$$ To obtain the local expression for this formula, one represents the automorphism $u$ by a family $(\ut:\Ut\to \GL_r(\C))_{\tau\in T}$ of smooth maps satisfying $\ut = g_{\ttp} \utp g_{\ttp}^{-1}$, from which one shows that $$\big(u(B))_{\tau} = \ut \Bt \ut^{-1} - (\db{\ut}) \ut^{-1}.$$ Note that the action of $\Aut(E)$ on $\Dol(E)$ is often denoted $$u(B)=B - (\db_B u)u^{-1}$$ in the literature, with a slight abuse of notation (writing simply $B$ for the operator $\db_B$, and $\db_B$ for $\db_{\mathrm{ad}B}$), which simplifies the practical computations. The fundamental result that we need about Dolbeault operators is the following proposition.
\begin{proposition}\label{hol_bundles_and_Dolbeault_op}
Let $E$ be a smooth complex vector bundles of rank $r$ and degree $d$ over the compact Riemann surface $M$. Then the set $\Vect_M^{hol}(r,d)$ of isomorphism classes of holomorphic vector bundles of rank $r$ and degree $d$ over $M$ is in bijection with the orbit space $\Dol(E)/\Aut(E)$.
\end{proposition}
\noindent The proof of this result is easy except for the fact that a Dolbeault operator $\db_B$ on $E$ has sufficiently many linearly independent sections of $E$ satisfying $\db_B s=0$ (one needs the sheaf of germs of such sections to be locally free of rank $r$ over the sheaf of holomorphic functions on $M$), for a proof of which we refer to \cite{DK}, subsection 2.2.2, page 50. On higher-dimensional compact K\"ahler manifolds, the proposition remains true if $E$ is a fixed smooth complex vector bundle and one considers the set of holomorphic vector bundles having the same Chern classes as $E$ on the one hand, and integrable Dolbeault operators on $E$ on the other hand.\\ Let us now endow $E$ with a Hermitian metric $h$ and recall the following observation from linear algebra (\cite{AB}, section 8, page 570). The map \begin{equation*}
\begin{array}{rcl} \zeroonematrices & \longrightarrow & \antiHermonematrices \\ \Bt & \longmapsto & \At := \Bt - \Bt^* \end{array}\end{equation*} is an isomorphism of \textit{real} vector spaces, whose inverse map sends $\At$ to its $\C$-antilinear part, $$\At^{0,1}(\cdot) = \frac{\At(\cdot) + i \At(i\cdot)}{2}.$$ If $(g_{\ttp}: \Ut\cap \Utp \to \U_r)_{\tau,\tau'}$ is a \textit{unitary} $1$-cocycle representing the Hermitian vector bundle $(E,h)$ and if the family $(\Bt)_{\tau\in T}\in\Omega^{0,1}(U_{\tau},\gl_r(\C))$ satisfies $$\Bt = g_{\ttp} \Btp g_{\ttp}^{-1} - (\db{g_{\ttp}}) g_{\ttp}^{-1},$$ then $g_{\ttp}^*=g_{\ttp}^{-1}$ and the family $(\At = \Bt - \Bt^*)_{\tau\in T}$ satisfies $$\At = g_{\ttp} \Atp g_{\ttp}^{-1} - (dg_{\ttp})g_{\ttp}^{-1}.$$ In other words, if the family $(\Bt)_{\tau\in T}$ defines a Dolbeault operator $\db_B$ on $E$, then the family $A:=(\At)_{\tau\in T}$ defines a \textit{unitary connection} $d_A$ on $(E,h)$, given locally by $$(d_A s)_{\tau} = d\st + \At\st.$$ This sets up a bijection between $\Dol(E)$ and the set $\conn$ of unitary connections on $(E,h)$ which can be expressed globally as follows. The Hermitian structure on $E$ enables one to define the operator $\del_{B^*}$, and the maps \begin{equation}\label{real_isom} \begin{array}{rcl} \Dol(E) & \longrightarrow & \conn \\ \db_{B} & \longmapsto & d_{B-B^*} = \db_{B} + \del_{B^*} \end{array}, \end{equation} and $$\begin{array}{rcl} \conn & \longrightarrow & \Dol(E) \\ d_{A} & \longmapsto &  \db_{A^{0,1}}\ = d_A^{\, 0,1} \end{array},$$ which are inverse to one another. Now, as $\dim_{\R} M =2$, the Hodge star of $M$ induces a complex structure on $\oneforms$, and the map (\ref{real_isom}) is actually $\C$-linear for that complex structure. This means that $\Dol(E)$ and $\conn$ are in fact isomorphic as \textit{complex} affine spaces. Moreover, there is a unique $\Aut(E)$-action on $\conn$ making that isomorphism equivariant. It is given explicitly by the formula $$u(A) = A - (\db_{A^{0,1}}u)u^{-1} + \big((\db_{A^{0,1}}u)u^{-1}\big)^*,$$ for all $u\in \Aut(E)$ and all $A\in\conn$. This action extends the natural action of the group $\G_{(E,h)}$ of unitary automorphisms of $(E,h)$. The group $\unitarygaugegp$ is commonly called the unitary gauge group. And indeed, if $u$ is unitary, then $u^*=u^{-1}$, and the above formula becomes $$u(A) = A + (d_A u)u^{-1}$$ (again with the slight abuse of notation that consists in writing $A$ for $d_A$, and $d_A$ for $d_{\mathrm{ad}A}$, which we shall do systematically from now on). Locally, this action is given by the formula $$\big(u(A)\big)_{\tau} = \ut\At\ut^{-1} - (d\ut)\ut^{-1}.$$ Note that the group $\Aut(E)$ of all complex linear automorphisms of $E$ is commonly denoted $\cxgaugegp$, and called the complex gauge group. We then reach the goal of this subsection, which is to recall the following result.
\begin{proposition}
Let $(E,h)$ be a smooth Hermitian bundle of rank $r$ and degree $d$ over the compact Riemann surface $M$. Then the set $\Vect_M^{hol}(r,d)$ of isomorphism classes of holomorphic vector bundles of rank $r$ and degree $d$ over $M$ is in bijection with the orbit space $\conn/ \cxgaugegp$.
\end{proposition}
\noindent One may observe that this bijection does not depend on the choice of the metric $h$ on $E$. Indeed, if another metric $h'$ is chosen, then $h'=\varphi h \varphi^*$ for some $\varphi\in \cxgaugegp$, and $A$ is $h$-unitary if, and only if, $\varphi A\varphi^{-1}$ is $h'$-unitary. The bijection between $\Dol(E)$ and $\conn$, however, does depend on the choice of the metric $h$.
On higher-dimensional compact K\"ahler manifolds, a Dolbeault operator $\db_{B}$ is integrable if, and only if, the corresponding unitary connection has type $(1,1)$ curvature, and the proposition above may be generalised by restricting one's attention to such connections.

\subsection{The Narasimhan-Seshadri-Donaldson theorem}

\subsubsection{Characterisation of the stable bundles}

We recall the definition of stable, polystable and semistable bundles, which originates in geometric invariant theory. For an arbitrary, non-zero complex vector bundle $(\calE\to M)$, we denote $\mu(\calE)$ the ratio $\frac{\degr(\calE)}{\rk(\calE)}$, called the \textbf{slope} of $\calE$.
\begin{definition}[Stable, polystable and semistable bundles]\label{def_stable_bundle}
A holomorphic vector bundle $\calE$ over a compact Riemann surface $M$ is called \textbf{stable} if, for any holomorphic subbundle $\calF$ which is neither $\{0\}$ nor $\calE$, one has $$\mu(\calF) < \mu(\calE).$$ $\calE$ is called \textbf{polystable} if it is a direct sum of stable bundles of slope $\mu(\calE)$. Finally, $\calE$ is called \textbf{semistable} if, for any holomorphic subbundle $\calF$ which is neither $\{0\}$ nor $\calE$, one has $$\mu(\calF) \leq \mu(\calE).$$
\end{definition}
\noindent We recall that the condition $\mu(\calF) < \mu(\calE)$ is equivalent to $\mu(\calE/\calF) > \mu(\calE)$, and likewise with large inequalities. Evidently, a stable bundle is both polystable and semistable. And as a matter of fact, polystable bundles are semistable. To see this last point, one first needs to observe that if $$0\to \calE_1 \to  \calE \to \calE_2\to 0$$ is a short exact sequence of holomorphic vector bundles \textit{having same slope}, then $\calE_1,\calE_2$ semistable implies $\calE$ semistable, again with the same slope. One then concludes by arguing that a polystable bundle is a split extension of semistable bundles with the same slope.\\ Since we have seen in the previous subsection that holomorphic vector bundles of fixed topological type over $M$ correspond bijectively to orbits of unitary connections on a fixed Hermitian bundle with that same topological type, it is natural to look for a characterisation of the stable bundles in terms of the corresponding orbits of unitary connections. This is exactly the content of Donaldson's formulation of a theorem of Narasimhan and Seshadri (see \cite{Don_NS}). Before stating that result, we recall that, since $M$ is a compact oriented Riemannian manifold of real dimension $2$, we may identify smooth $2$-forms on $M$ with values in a vector bundle to global sections of that bundle ($0$-forms), using the Hodge star of $M$. In particular, if $F_A\in\antiHermtwoforms$ denotes the curvature of a unitary connection $A$ on $(E,h)$, then $*F_A$ is an element of $\antiHermsections$. We also recall that a holomorphic vector bundle is called \textit{indecomposable} if it cannot be written as a proper direct sum.
\begin{theorem}[Donaldson, \cite{Don_NS}]\label{charac_stable_bundles}
A holomorphic vector bundle $\calE$ of rank $r$ and degree $d$ is stable if, and only if, it is indecomposable, and admits a Hermitian metric $h$ and a compatible unitary connection $A$, whose curvature $F_A$ satisfies $$*F_A = i2\pi\frac{d}{r}\ \Id_{\calE} \in \Omega^0(M;\mathfrak{u}(\calE,h)).$$ Such a connection is then unique up to a unitary automorphism of $(\calE,h)$.
\end{theorem}
\noindent A holomorphic vector bundle is therefore polystable if, and only if, it admits a unitary connection of the type above (such a connection is called a \textbf{Yang-Mills connection}). We shall often denote $\mu= \frac{d}{r}$ the slope of a holomorphic vector bundle of rank $r$ and degree $d$.\\ Let now $(E,h)$ be a fixed smooth Hermitian vector bundle of rank $r$ and degree $d$ over $M$. The map $$F: \begin{array}{rcl} \conn & \longrightarrow & \twoforms \\ A & \longmapsto & F_A:= d_A\circ d_A\end{array}$$ taking a unitary connection to its curvature is equivariant for the action of $\G_{(E,h)}$ on $\conn$ considered in the previous subsection and the conjugacy action of $\G_{(E,h)}$ on $\twoforms$. The Hodge star $$*: \begin{array}{rcl} \twoforms &\longrightarrow & \sections\\ R & \longmapsto & s\ \mathrm{such\ that}\ R=s\, \vol_M\end{array}$$ is also equivariant for the conjugacy action of the gauge group, so, as $\tmu$ is a central element in $\sections$, the fibre $(*F)^{-1}(\{\tmu\})$ is $\G_{(E,h)}$-invariant.
\begin{corollary}
The set of gauge equivalence classes of polystable holomorphic vector bundles of rank $r$ and degree $d$ over $M$ is in bijection with $$(*F)^{-1}(\{\tmu\}) / \G_{(E,h)}.$$
\end{corollary}
 
\subsubsection{The K\"ahler structure of moduli spaces of semistable bundles}

Recall that we have denoted $\Mod$ the moduli variety of semistable holomorphic vector bundles of rank $r$ and degree $d$ over $M$. As every semistable bundle of slope $\mu$ admits a Jordan-H\"older filtration whose successive quotients are stable bundles of slope $\mu$ (see for instance \cite{Le_Potier_ENS}, expos\'e 2, theorem I.2, page 33), the associated graded vector bundle is a polystable bundle of rank $r$ and degree $d$, whose isomorphism class, as a graded vector bundle, is independent of the choice of the filtration. Next, by a theorem of Seshadri (\cite{Seshadri}), two semistable holomorphic bundles define the same point in $\Mod$ if, and only if, the associated graded bundles are isomorphic. This theorem, combined with Donaldson's characterisation of polystable bundles, establishes a bijection between $\Mod$ and the set of gauge equivalence classes of polystable bundles of rank $r$ and degree $d$, which we have just expressed as $$(*F)^{-1}(\{i2\pi\frac{d}{r}\ \mathrm{\Id}_E\})/\G_{(E,h)},$$ where $F:\conn\to\twoforms$ is the curvature and $*$ is the Hodge star of $M$. We may also write this $$\Mod = F^{-1}(\{*i2\pi\frac{d}{r}\ \mathrm{\Id}_E\})/\G_{(E,h)}.$$ The $\Ad$-invariant positive definite inner product $(M,N)\mapsto -\tr(MN)$ on $\fraku_r$ induces a gauge-invariant Riemannian metric on the bundle $\antiHermendom$ of anti-Hermitian endomorphisms of $(E,h)$. This in turn induces a $\G_{(E,h)}$-invariant positive definite scalar product on the real vector space $\sections$, defined by $$(s,t) \mapsto \int_M -\tr(st)\, \vol_M.$$ Hence a $\G_{(E,h)}$-equivariant isomorphism of real vector spaces $$\begin{array}{rcl}\sections & \longrightarrow & \big(\sections\big)^*\\ t & \longmapsto & \left( s\mapsto \int_M -\tr(st)\, \vol_M\right)\end{array}.$$ Composing the Hodge star $*:\twoforms \to \sections$ with this isomorphism gives a $\G_{(E,h)}$-equivariant isomorphism of real vector spaces $$\begin{array}{rcl} \twoforms & \overset{\simeq}{\longrightarrow} & (\sections)^* \\ R & \longmapsto & \left( s \mapsto \int_M -\tr(sR)\right) \end{array},$$ so we may think of the curvature $F$ as a map with values in $(\sections)^*$, the dual of the Lie algebra of the gauge group. Atiyah and Bott have shown (\cite{AB}) that $\conn$ has a K\"ahler structure and that the action of $\G_{(E,h)}$ is a K\"ahler action which admits the curvature as a momentum map. Therefore, the above presentation of $\Mod$ gives this moduli space a K\"ahler structure, obtained by performing reduction at level $*i2\pi\frac{d}{r} \, \mathrm{\Id}_E$ on the infinite-dimensional K\"ahler manifold $\conn$. We shall sometimes denote $$\quot$$ the quotient $F^{-1}(\{*i2\pi\frac{d}{r}\, \mathrm{\Id}_E\})/\G_{(E,h)}$. Note that we should in fact consider $L^2_1$ connections instead of $C^{\infty}$ ones, with curvature in $L^2$ and gauge transformations in $L^2_2$, in order to make the affine space $\conn$ a Banach manifold (see \cite{AB}, sections $6$ and $14$, and \cite{Don_NS}, section $2$). We have deliberately omitted this, to lighten the exposition and stress the geometric ideas underlying the construction.\\ The only part of the above theory that we need to make explicit in order to use it later in the paper is the expression of the symplectic form of $\conn$, which we denote $\w$. As the space of all unitary connections is an affine space whose group of translations is $\oneforms$, the tangent space at a given point $A$ of $\conn$ may be canonically identified to $\oneforms$. If $a,b\in\oneforms$ are two tangent vectors at $A$, the $2$-form $a\wedge b$ is $\antiHermendom \otimes \antiHermendom$-valued. Combining this with the $\Ad$-invariant scalar product $M\otimes N \mapsto -\tr(MN)$ of $\fraku_r$, one obtains a real valued $2$-form on $M$, denoted $-\tr(a \wedge b)$. Explicitly, it is the $2$-form on $M$ defined pointwise by $$\big(-\tr(a\wedge b)\big)_x: \begin{array}{rcl} T_x M \times T_x M & \longrightarrow & \R \\ (v_1,v_2) & \longmapsto & -\frac{1}{2} \tr\big( a_x(v_1)b_x(v_2) - a_x(v_2)b_x(v_1)\big)\end{array}.$$ Then, the expression $$\w_A(a,b) = \int_M -\tr(a \wedge b)$$ defines a $2$-form $\w$ on $\conn$, and this $2$-form is symplectic.

\section{Real semistable bundles over a Riemann surface}\label{lag_embedding}

In this section, we show that, for any anti-holomorphic involution $\sigma$ of $M$, the set $$\calN:=\{[\gr(\calE)]\in\Mod : \calE\ \mathrm{is\ semistable\ and\ real}\}$$ of moduli of real semistable bundles of rank $r$ and degree $d$ is a totally real, totally geodesic, Lagrangian submanifold of $\Mod$. Our strategy is to show that there is an anti-symplectic, involutive isometry $\bs$ of $\Mod$ such that $\calN$ is a union of connected components of $\Fix(\bs)$. To that end, we use the presentation of the moduli variety $\Mod$ as a K\"ahler quotient, $$\Mod = \quot,$$ and show that $\bs$ comes from a family of anti-symplectic, involutive isometries $\asigt$ of $\conn$, where $\sigt$ runs through the set of real Hermitian structures of $(E,h)$.\\ Elements of $\calN$ are, by definition, strong equivalence classes of semistable holomorphic vector bundles of rank $r$ and degree $d$ which contain a real bundle. As a matter of fact, more is true: any bundle in the strong equivalence class of a bundle which is both semistable and real is itself real. Let us show this. By Seshadri's theorem, the strong equivalence class of a semistable holomorphic vector bundle $\calE$ is the graded isomorphism class of the graded vector bundle $\gr(\calE)$ associated to any Jordan-H\"older filtration of $\calE$. The successive quotients of that filtration are stable bundles of slope equal to that of $\calE$, and $\gr(\calE)$ is therefore a polystable bundle of slope $\mu(\calE)$. Assume now that $\calE$ is both semistable and real. Then, for any Jordan-H\"older filtration of $\calE$, the successive quotients are both stable and real (and have same slope). Indeed, this is a simple consequence of the fact that the kernel and the image of a homomorphism of real bundles (a bundle map which intertwines the real structures) have naturally induced real structures, and this in fact turns the category of real semistable bundles of slope $\mu$ into an Abelian category  which is stable by extensions (a strict subcategory of the Abelian category of semistable bundles of slope $\mu$, compare \cite{Le_Potier_ENS}). As a consequence, the graded bundle associated to a real semistable bundle is a direct sum of bundles which are both stable and real, and any bundle $\calE$ such that $\gr(\calE)$ is a direct sum of bundles which are both stable and real is itself real. The next result then shows that points of $\calN$ are precisely graded real isomorphism classes of such bundles. This in turn suggests that points of $\calN$ should, perhaps, be thought of as moduli of real semistable bundles over $(M,\sig)$.
\begin{proposition}\label{real_moduli}
Let $(\calE,\sigt)$ and $(\calE',\sigt')$ be two holomorphic bundles which are both semi\-stable and real, and assume that $\gr(\calE)$ is isomorphic to $\gr(\calE')$ as a graded vector bundle. Then $\gr(\calE)$ is isomorphic to $\gr(\calE')$ as a graded real vector bundle.
\end{proposition}
\begin{proof}
One first proves the result under the additional assumption that $\calE$ and $\calE'$ are both stable, and conclude by induction on the length of the Jordan-H\"older filtration. For a stable bundle $\calE$, one has $\gr(\calE)\simeq \calE$, so the assumption of the proposition is that $\calE\simeq \calE'$. Replacing $\sigt'$ with $\varphi\sigt'\varphi^{-1}$ if necessary, we may further assume that $\sigt$ and $\sigt'$ are two distinct real structures on the same vector bundle $\calE$. Then $\sigt\sigt'$ is $\C$-linear and, as $\calE$ is stable, this implies that $\sigt\sigt'=\lambda\in\C^*$. This in turn implies that $\sigt = |\lambda|^2 \sigt$, so $\lambda = e^{i\theta}$ for some $\theta \in \R$, whence one obtains $$\sigt = e^{i\theta} \sigt' = e^{i\frac{\theta}{2}} \sigt' (e^{-i\frac{\theta}{2}}\cdot),$$ showing that $\sigt$ and $\sigt'$ are conjugate by an automorphism of $\calE$.
\end{proof}
\noindent As a final remark on $\calN$, we observe that the functor $\calE \mapsto \os{\calE}$ sends a Jordan-H\"older filtration of $\calE$ to a Jordan-H\"older filtration of $\os{\calE}$, so it induces an automorphism $$\bs: [\gr(\calE)] \mapsto [\gr(\os{\calE})]$$ of the moduli variety $\Mod$, which is involutive and fixes $\calN$ pointwise. We now set out to prove that this involution is an anti-symplectic isometry of $\Mod$, and that $\calN$ is a union of connected components of $\Fix(\bs)$. Note that, since the tangent space to $\Mod$ at a given smooth point $[\gr(\calE)]$ may be identified with $H^1(M;\End(\calE))$, one sees that the involution $\bs$ is anti-holomorphic (for the tangent map takes an $\End(\calE)$-valued holomorphic $1$-form $\nu$ to $\os{\nu}$). We shall show:
\begin{enumerate}
\item that there is an involution $\asigt$ associated to each choice of a real Hermitian structure $\sigt$ on $(E,h)$, such that a unitary connection $A\in\conn$ defines a real holomorphic structure on $(E,h,\sigt)$ if, and only if, $\asigt(A)=A$.
\item that $\asigt$ is an anti-symplectic isometry of $\conn$ which induces $\bs$ on the K\"ahler quotient $$\fibre/ \unitarygaugegp = \Mod,$$ confirming the fact that the latter involution is an anti-symplectic isometry.
\item that we can form a so-called \textbf{real quotient}, $$\calL := \Big(\fibre\Big)^{\asigt} / \realgaugegp,$$ which embeds onto a union of connected component of $\calN\subset \Mod$. As $\asigt$ induces $\bs$, the real quotient also embeds onto a union of connected components of $\Fix(\bs)$, and our result will be proved.
\end{enumerate}

\subsection{Real unitary connections}

Let $(E,h,\sigt)$ be a fixed real Hermitian bundle (meaning that $\sigt$ is a $\C$-antilinear isometry which covers $\sigma$ and squares to $\Id_E$). The choice of $\sigt$ induces a canonical isomorphism $\varphi: \os{E} \overset{\simeq}{\longrightarrow} E$, as well as so-called \textbf{real invariants} which classify $(E,h,\sigt)$ up to isomorphism of real Hermitian bundles. We denote $\Ms$ the fixed-point set of $\sigma:M\to M$, and $g$ the genus of $M$.
\begin{proposition}[\cite{BHH}, Propositions 4.1 and 4.2]\label{real_invariants}One has:
\begin{enumerate}
\item if $\Ms = \emptyset$, then real Hermitian bundles are topologically classified by their rank and degree. It is necessary and sufficient for a real Hermitian bundle of rank $r$ and degree $d$ to exist that $d$ should be even.
\item if $\Ms \not= \emptyset$, then $(\Es \to \Ms)$ is a real vector bundle in the ordinary sense, over the disjoint union $\Ms = \calC_1 \sqcup \ldots \sqcup \calC_k$ of at most $(g+1)$ circles. Denoting $w^{(j)} := w_1(\Es_j) \in H^1(S^1; \Z / 2\Z) \simeq \Z / 2\Z$ the first Stiefel-Whitney class of $\Es$ restricted to $\calC_j$, real Hermitian bundles over $M$ are topologically classified by their rank, their degree, and the sequence $(w^{(1)}, \ldots, w^{(k)})$. It is necessary and sufficient for a real Hermitian bundle with given invariants $r$, $d$ and $(w^{(1)}, \ldots, w^{(k)})$ to exist that $$w^{(1)} + \cdots + w^{(k)} \equiv d\ (\mod 2).$$
\end{enumerate}
\end{proposition}

\noindent The choice of a real structure $\sigt$ on $(E,h)$ induces real structures on the complex vector space of smooth global sections of $E$,

\begin{eqnarray*}
\sectionsofE & \longrightarrow &  \sectionsofE \\
s & \longmapsto & \ov{s}: x \mapsto \ov{s(\ov{x})},
\end{eqnarray*}

\noindent (the fixed-points of which are the real sections of $E$, defined in section \ref{intro}) and, more generally, on the space of $E$-valued $k$-forms on $M$,

\begin{eqnarray*}
\kformsofE & \longrightarrow &  \kformsofE \\
\eta & \longmapsto & \ov{\eta}: v\mapsto \ov{\eta_{\ov{x}} (\ov{v})}.
\end{eqnarray*}

\begin{definition}[Real unitary connections]\label{real_unitary_connection_def}
A unitary connection $$d_A : \sectionsofE \longrightarrow \oneformsofE$$ is called \textbf{real} if it commutes to the real structures of $\sectionsofE$ and $\oneformsofE$: $$d_A\ov{s} = \ov{d_A s}\ \mathrm{for\ all\ } s\in \sectionsofE.$$ 
\end{definition}

\noindent A similar definition is possible on Dolbeault operators: the complex vector space $\zerooneforms$ is invariant under the real structure of $\oneformsEzero$, and a Dolbeault operator $$\db_B: \sectionsofE \longrightarrow \zerooneforms$$ is called real if it commutes to the real structures. This definition is compatible with the isomorphism between $\Dol(E)$ and $\conn$ in the sense that $d_{A}^{\, †0,1}$ is real if, and only if, $d_A$ is real. In particular, $\ker d_A^{\, 0,1}$ has an induced real structure, and the holomorphic bundle $(E,d_A^{\, 0,1})$ is a real holomorphic bundle in the sense of Atiyah. Conversely, the Dolbeault operator of a real holomorphic bundle commutes to the real structures of $\sectionsofE$ and $\zerooneforms$, so the unitary connection associated to that operator is a real unitary connection in the sense of definition \ref{real_unitary_connection_def}.\\ We now come to the construction of the involution $\asigt$ of $\conn$. To define it, we make use of the canonical isomorphism $$\phi:\os{E} \overset{\simeq}{\longrightarrow} E$$ determined by $\sigt$, and define $\asigt(A)$ to be the unitary connection on $E$ such that, for all $s\in\sectionsofE$, $$d_{\asigt(A)} s = \phi \big(d_{\os{A}} (\phi^{-1} s)\big).$$ To make this more explicit, let us choose a unitary $1$-cocycle $$(\gttp:\Uttp \to \U_r)_{\tau,\tau'}$$ representing $E$ and subordinate to a covering $(\Ut)_{\tau\in T}$ by open sets which satisfy $\sig(\Ut)=\Ut$. Then $\phi$ may be represented by a $0$-cocycle $(\Lt:\Ut\to\U_r)_{\tau\in T}$ such that $$\Lt \os{\gttp}\Lt^{-1} = \gttp \quad \mathrm{and}\quad \os{\Lt}=\Lt^{-1}.$$ The first condition translates the fact that $\phi$ is an isomorphism between $\os{E}$ and $E$, and the second condition translates the fact that $\sigt^2 = \Id_E$. It can then easily be checked that, if $(\At: \Ut \to \antiHermonematrices)_{\tau\in T}$ is a framed unitary connection on $E$, then so is $$\big(\Lt\os{\At}\Lt^{-1} - (d\Lt)\Lt^{-1}\big)_{\tau\in T},$$ and the connection it defines is $\asigt(A)$. In the following, we shall simply denote
$$\asigt: \begin{array}{rcl} \conn & \longrightarrow & \conn \\ A & \longmapsto & \ov{A} \end{array}.$$ This suggestive notation is justified by the following result.

\begin{proposition}\label{charac_real_connections}
The unitary connection $$d_A: \sectionsofE \longrightarrow \oneformsofE$$ is real if, and only if, $\ov{A}=A.$
\end{proposition}

\begin{proof}
The key observation is that $$d_{\ov{A}}\, s = \ov{d_A \ov{s}},$$ which follows from the definition of $d_{\ov{A}}$. Since, by definition, the connection $d_A$ is real if, and only if, $\ov{d_A \ov{s}}= d_A s$ for all $s$, the proposition is proved.
\end{proof}

\noindent The result may also be proved by computing in a local frame of the type described above. If we choose a different real Hermitian structure $\sigt'$ on $(E,h)$, we obtain a new involution $\alpha_{\sigt'}$, whose fixed points also define real holomorphic bundles, but with possibly different real invariants. Thus, letting $\sigt$ run through the set of all possible real Hermitian structures, we obtain, as fixed-point sets of the various involutions thus defined, the unitary connections that define all possible real holomorphic bundles of rank $r$ and degree $d$. As a matter of fact, it suffices to let $\sigt$ run through the set of all topological types of real Hermitian structures of rank $r$ and degree $d$, and proposition \ref{real_invariants} shows that these come in finite number.

\subsection{Properties of the involutions}

We now study the properties of $\asigt$ for a fixed $\sigt$. We start by observing that the choice of $\sigt$ induces an involution of the unitary gauge group $\unitarygaugegp$, and of the vector space $\antiHermtwoforms$, which may be viewed as the dual as the dual of the Lie algebra of the unitary gauge group. These involutions are

\begin{eqnarray*}
\unitarygaugegp & \longrightarrow & \unitarygaugegp \\
u & \longmapsto & \phi\, \os{u}\, \phi^{-1}
\end{eqnarray*}

\noindent and

\begin{eqnarray*}
\antiHermtwoforms & \longrightarrow & \antiHermtwoforms \\
R & \longmapsto & \phi\, \os{R}\, \phi^{-1}
\end{eqnarray*}

\noindent We shall simply denote them $u\mapsto \ov{u}$ and $R\mapsto\ov{R}$.

\begin{proposition}
One has the following compatibility relations:
\begin{enumerate}
\item between $\asigt$ and the gauge action: $$\ov{u(A)} = \ov{u}(\ov{A}).$$
\item between $\asigt$ and the momentum map of the gauge action: $$F_{\ov{A}} = \ov{F_A}.$$
\end{enumerate}
\noindent Moreover, $\asigt$ is an anti-symplectic involutive isometry of $\conn$.
\end{proposition}

\noindent Note that $\ov{F_A}$ is, by definition, the operator from $\sectionsofE$ to $\twoformsofE$ that takes a section $s$ to the $2$-form $\ov{F_A\ov{s}}$. We could have defined the operator $\ov{d_A}$ in a similar way, and proposition \ref{charac_real_connections} then shows that $d_A$ is real if, and only if, $\ov{d_A} = d_A$.
\begin{proof}
First, as $\os{\phi} = \phi^{-1}$, $\asigt$ squares to the identity. Second, observe that there is an involution

$$\begin{array}{rcl}
\antiHermoneforms & \longrightarrow & \antiHermoneforms \\
a & \longmapsto & \ov{a} := \phi\, \os{a}\, \phi^{-1}
\end{array}.$$

\noindent Then, for all $A\in\conn$ and all $a,b\in T_A \conn = \antiHermoneforms$,

\begin{eqnarray*}
(\asigt^*\w)_A (a,b) & = & \int_M -\tr(\ov{a} \wedge \ov{b}) \\
& = & \int_M \sig^*\big(-\tr(a \wedge b)\big) \\
& = & -\int_M -\tr(a \wedge b)\\
& = & - \w_A (a,b)
\end{eqnarray*}

\noindent As a consequence, to show that $\asigt$ is an isometry, it suffices to show that $\asigt$ is $\C$-antilinear with respect to the compatible complex structure of $\conn$, which is given by the Hodge star. This last point follows from the fact that the tangent map to $\asigt$ is $$a \mapsto \ov{a}\ \mathrm{on}\ \antiHermoneforms,$$ which is clearly $\C$-antilinear. Finally, let us prove that we have the asserted compatibility relations.
\begin{enumerate}
\item One has $$\ov{u(A)} = \ov{A + (d_A u)u^{-1}} = \ov{A} + (\ov{d_A u})\ov{u^{-1}} = \ov{A} + (d_{\ov{A}} \ov{u}) \ov{u}^{-1} = \ov{u}(\ov{A}).$$
\item Similarly, for all $s\in\sectionsofE$, $$ F_{\ov{A}} s = d_{\ov{A}} (d_{\ov{A}} s) = d_{\ov{A}} \ov{(d_{A} \ov{s})} = \ov{d_{A} (d_A \ov{s})} = \ov{F_A \ov{s}},$$ hence $F_{\ov{A}} = \ov{F_A}$.
\end{enumerate}
\end{proof}

\subsection{Embedding the real quotients}

For a given $\sigt$ on $(E,h)$, the elements of $\Fix(\asigt)$ are the connections that define the real holomorphic structures on $E$ with real invariants those determined by $\sigt$. Let us denote $\realgaugegp$ the subgroup of $\unitarygaugegp$ consisting of fixed points of the automorphism $u \mapsto \ov{u}$, and call it the \textbf{real gauge group}. Note that it is precisely the group of automorphisms of the real Hermitian bundle $(E,h,\sigt)$. The compatibility relation $\ov{u(A)} = \ov{u}(\ov{A})$ shows that the real gauge group acts on the space $\Fix(\asigt)$ of real connections. We then observe that $*i2\pi\frac{d}{r}\Id_E$ is fixed under the involution $R\mapsto \ov{R}$ of $\antiHermtwoforms$. Indeed,
\begin{eqnarray*}
\phi\, \os{(i2\pi\frac{d}{r}\Id_E \vol_M)}\, \phi^{-1} & = & \left( \phi (-i2\pi\frac{d}{r} \Id_{\os{E}}) \phi^{-1} \right) \big(\sigma^* \vol_M\big) \\
& = &  (- i2\pi\frac{d}{r}\Id_E)\, (- \vol_M) \\
& = & i2\pi\frac{d}{r}\Id_E \vol_M,
\end{eqnarray*} and so $\fibre$ is invariant under $\asigt$. Consequently, the following quotient, $$\calL=(\fibre)^{\asigt} / \realgaugegp,$$ is a well-defined object. By Donaldson's theorem and proposition \ref{charac_real_connections}, its elements are real gauge equivalences classes of holomorphic bundles of rank $r$ and degree $d$ which are both polystable and real. The compatibility of $\asigt$ with the gauge action and the momentum map of that action also shows that $\asigt$ induces an involution of the quotient $$\Mod = \fibre / \unitarygaugegp,$$ sending the gauge equivalence class of a connection $A$ to that of $\ov{A}$. In other words, the involution induced by $\asigt$ is $$\bs: [\gr(\calE)] \mapsto [\gr(\os{\calE})],$$ \textit{regardless of the choice of the Hermitian real structure} $\sigt$ \textit{on} $(E,h)$. In particular, $\bs$ is an anti-symplectic, involutive isometry of $\Mod$. Evidently, there is a map $$\calL \longrightarrow \calN \subset \Fix(\bs),$$ of which we can think as the map which forgets the real structure of a holomorphic bundle which is both polystable and real. Observe that the real quotient $\calL$ has dimension half the dimension of the K\"ahler quotient $\Mod$, and that so does $\Fix(\bs)$. Therefore, to prove that $\calN$ is a union of connected components of $\Fix(\bs)$, it suffices to prove that the above mapping is a closed embedding of $\calL$ into $\calN$.

\begin{proposition}
Let $A$ and $A'$ be two real, irreducible, Yang-Mills connections which lie in the same $\unitarygaugegp$-orbit. Then they lie in the same $\realgaugegp$-orbit.
\end{proposition}

\noindent We recall that an irreducible connection is a unitary connection that defines an \textit{indecomposable} holomorphic structure on $(E,h)$. This is equivalent to asking that its stabiliser, \textit{in} $\cxgaugegp$, should be isomorphic to $\C^*$.

\begin{proof}
The proof is similar to that of proposition \ref{real_moduli}. Assume that $A'=u(A)$ for some $u\in\unitarygaugegp$. Then $$u(A)=A'=\ov{A'}=\ov{u(A)}=\ov{u}(\ov{A})=\ov{u} (A).$$ As $A$ is both Yang-Mills and irreducible, its stabiliser is, by Donaldson's theorem, contained in $S^1$. This implies that $u^{-1}\ov{u} = e^{i\theta}$ for some $\theta\in\R$. Put $v=e^{i\frac{\theta}{2}} u$. Then $v(A)=u(A)=A'$, and $\ov{v}=e^{-i\frac{\theta}{2}} \ov{u} = e^{-i\frac{\theta}{2}} e^{i\theta} u = e^{i\frac{\theta}{2}} u = v$, so $v\in\realgaugegp$.
\end{proof}

\begin{corollary}
The map $$\calL\longrightarrow\calN,$$ which takes the real gauge orbit of a real Yang-Mills connection to its unitary gauge orbit, is injective.
\end{corollary}

\begin{proof}
Recall that a Yang-Mills connection is an element of $\fibre$, and that it is a direct sum of irreducible Yang-Mills connections. If the connection is real, then so are its irreducible components and the corollary follows by applying the proposition to those irreducible components.
\end{proof}

\noindent We may now state the conclusion to this paper.

\begin{theorem}\label{lag_submanifold}
The set $$\calN = \{ [\gr(\calE)]\in\Mod : \calE\ \mathrm{is\ semistable\ and\ real}\}$$ of moduli of holomorphic bundles of rank $r$ and degree $d$ which are both semistable and real is a totally real, totally geodesic, Lagrangian submanifold of $\Mod$.
\end{theorem}

\begin{ack}
This research was carried out at the University of Los Andes in Bogot\'a, and at the IHES in Bures-sur-Yvette, over the second half of 2009. I thank both these institutions for their hospitality. I would also like to thank Olivier Guichard, Nan Kuo Ho, Melissa Liu, and Richard Wentworth, for discussing the constructions of the paper with me and for asking challenging questions. Finally, I thank Thomas Baird for his comments on an early version, and the referee for comments which have helped improve the general presentation of the paper.
\end{ack}


\end{document}